\documentclass[11pt,reqno]{amsart}
\usepackage{geometry}               
\usepackage{amsmath}
\usepackage{amssymb}
\usepackage{amsthm}
\usepackage{amsxtra}
\usepackage{amscd, amsrefs}
\usepackage{mathtools}
\usepackage[all]{xy}
\usepackage{blindtext}
\usepackage[utf8]{inputenc}
\usepackage[english]{babel}

\usepackage{tikz}

\usepackage{microtype}

 \usepackage{hyperref}
\theoremstyle{plain}
\newtheorem{thm}{Theorem}[section]
\newtheorem{lem}[thm]{Lemma}
\newtheorem{prop}[thm]{Proposition}
\newtheorem{cor}[thm]{Corollary}

\theoremstyle{definition}
\newtheorem{defn}[thm]{Definition}

\theoremstyle{remark}
\newtheorem*{rem}{Remark}

\numberwithin{equation}{section}

\begin{document}
\date{\today}

\title
[Global Kato type smoothing estimates via local ones]{Global Kato type smoothing estimates via local ones for dispersive equations}

\author[J. Lee]{Jungjin Lee}

\address{Department of Mathematical Sciences, School of Natural Science, Ulsan National Institute of Science and Technology, UNIST-gil 50, Ulsan 44919, Republic of Korea}
\email{jungjinlee@unist.ac.kr}




\thanks{
The author was supported in parts by NRF grant No. 2017R1D1A1B03036053 (Republic of Korea).}

\begin{abstract}
In this paper we show that the local Kato type smoothing estimates are essentially equivalent to the global Kato type smoothing estimates for some class of dispersive equations including the Schr\"odinger equation. From this we immediately have two results as follows. 
One is that the known local Kato smoothing estimates are sharp. The sharp regularity ranges of the global Kato smoothing estimates are already known, but those of the local Kato smoothing estimates are not. Sun, Tr\'elat, Zhang and Zhong \cite{sun2017sharpness} have shown it only in spacetime $\mathbb R \times \mathbb R$. Our result resolves this issue in higher dimensions. The other one is the sharp global-in-time maximal Schr\"odinger estimates. 
Recently, the pointwise convergence conjecture of the Schr\"odinger equation has been settled by Du--Guth--Li--Zhang \cite{du2018pointwise} and Du--Zhang \cite{du2018sharp}. For this they proved related sharp local-in-time maximal Schr\"odinger estimates. By our result, these lead to the sharp global-in-time maximal Schr\"odinger estimates.
\end{abstract}

\maketitle

\section{Introduction}

Fix $n \ge 1$ and $m > 1$. Let $\Phi \in C^\infty(\mathbb R^n \setminus 0)$ be a real-valued function satisfying the following conditions:
\begin{equation} \label{homog_cond}
\left\{
\begin{aligned}
&|\nabla \Phi(\xi)| \neq 0 \quad \text{ for all $\xi \neq 0$}, \\
&\Phi(\lambda \xi) = \lambda^{m} \Phi(\xi) \quad \text{ for all $\lambda >0$ and $\xi \neq 0$}.
\end{aligned}
\right.
\end{equation}
We are concerned with the solutions $u$ to equations 
\begin{equation} \label{gen_form} 
\left\{ 
\begin{aligned}
i \partial_t u(t,x) +  \Phi(D)u(t,x) &=0 \quad \text{ in } \mathbb R \times \mathbb R^n,\\
u(0,x) &= f(x) \quad \text{ in } \mathbb R^n,
\end{aligned}
\right.
\end{equation}
where  
$\Phi(D)$ is the corresponding Fourier multiplier to the function $\Phi$, that is, $\Phi(D)$ is defined by $\Phi(D) f = (\Phi \hat f)^{\vee}$.

For a function $f$ on $X \times Y$, we define the mixed norm $\| f \|_{L_x^q(X;L_y^r(Y))}$ by
\[
\| f \|_{L_x^q(X;L_y^r(Y))} := \bigg( \int_{X} \Big( \int_{Y} |f(x,y)|^r dy\Big)^{q/r} dx \bigg)^{1/q}.
\]
Let $\mathbb B^n$ be the unit ball in $\mathbb R^n$ and $\mathbb I=\{t\in \mathbb R : 1/2 \le t \le 2 \}$ be an interval.
For $1 \le  q, r \le \infty$ and $\alpha \in \mathbb R$ we use the notations $K_{loc}(L_x^q L_t^r; \alpha)$ and $K_{loc}(L_t^r L_x^q ; \alpha)$ to denote the \textit{local-in-time} smoothing estimates
\begin{align*}
\| \langle D \rangle^{\alpha} u \|_{L_x^q(\mathbb B^n; L_t^r(\mathbb I))} \le C_{\alpha,m}\|f\|_{L^{2}(\mathbb R^n)} 
\intertext{and} 
\| \langle D \rangle^{\alpha} u \|_{L_t^r(\mathbb I; L_x^q(\mathbb B^n))} \le C_{\alpha,m}\|f\|_{L^{2}(\mathbb R^n)}
\end{align*} 
for all $f \in L^2(\mathbb R^n)$ respectively, where  $\langle D \rangle$ is the operator associated with a symbol $\langle \xi \rangle := (1+|\xi|^2)^{1/2}$. Similarly we use $K_{glb}(L_x^q L_t^r;\alpha)$ and $K_{glb}(L_t^r L_x^q ;\alpha)$ to denote the \textit{global-in-time} smoothing estimates
\begin{align*} 
\| \langle D \rangle^{\alpha} u \|_{L_x^q(\mathbb B^n; L_t^r(\mathbb R))} \le C_{\alpha,m}\|f\|_{L^{2}(\mathbb R^n)}
\intertext{and} 
\| \langle D \rangle^{\alpha} u \|_{L_t^r(\mathbb R; L_x^q(\mathbb B^n))} \le C_{\alpha,m}\|f\|_{L^{2}(\mathbb R^n)}
\end{align*} 
for all $f \in L^2(\mathbb R^n)$ respectively. For convenience we denote by $K_{loc}(L^q_{t,x};\alpha) := K_{loc}(L^q_{t}L^q_{x};\alpha)$ and $K_{glb}(L^q_{t,x};\alpha) := K_{glb}(L^q_{t}L^q_{x};\alpha)$.

In this paper we show that the local-in-time smoothing estimates $K_{loc}(L_x^q L_t^r; \alpha)$,and $K_{loc}(L_t^rL_x^q ; \alpha)$ are essentially equivalent to the global-in-time  smoothing estimates $K_{glb}(L_x^q L_t^r; \alpha)$ and $K_{glb}(L_t^r L_x^q ; \alpha)$, respectively. 

\begin{thm} \label{thm:LocToGlb}
Let $2 \le q,r < \infty$. Suppose that $\Phi$ satisfies the condition \eqref{homog_cond}. 
\begin{itemize}
\item[(i)]
The local-in-time smoothing estimate $K_{loc}(L_x^qL_t^r;\alpha)$ implies
the global-in-time smoothing estimate $K_{glb}(L_x^{q}L_t^{\tilde r}; \alpha - \delta)$ for all $\tilde r >r$ and $\delta > n(\frac{1}{r}- \frac{1}{\tilde r})$. 
\item[(ii)]
The local-in-time smoothing estimate $K_{loc}(L_t^rL_x^q;\alpha)$ implies
the global-in-time smoothing estimate $K_{glb}(L_t^{\tilde r}L_x^{q};\alpha - \delta)$ for all $\tilde r >r$ and $\delta > n(\frac{1}{r}- \frac{1}{\tilde r})$. 

\end{itemize}
\end{thm}

The most basic estimate for the solution $u$ to \eqref{gen_form} is the energy identity that for any $t \in \mathbb R$,
\begin{equation*} 
\|u(t)\|_{L^2(\mathbb R^n)} = C\|f\|_{L^2(\mathbb R^n)}
\end{equation*}
for all $f \in L^2(\mathbb R^n)$,
where $u(t) = u(t,\cdot)$.
It implies that the solution $u$ has as much regularity as the initial $f$ in $L^2$-space. 
Kato \cite{kato1983cauchy} first observed that an integration locally in spacetime makes the solution $u$ smoother than the initial $f$, and showed the local smoothing estimate $K_{loc}(L^2_{t,x};1)$ for the KdV equation. Later, Constantin--Saut \cite{constantin1988local}, Sj\"olin \cite{sjolin1987regularity} and Vega \cite{vega1988schrodinger} independently proved the local smoothing estimates $K_{loc}(L_{t,x}^{2};\alpha)$ for some class of dispersive equations including the Schr\"odinger equation. 

In \cite{vega1988schrodinger}, Vega  proved the global  smoothing estimate $K_{glb}(L_{t,x}^2;\alpha)$, $\alpha < 1/2$ for the Schr\"odinger equation in all dimensions, and the endpoint estimate $K_{glb}(L_{t,x}^2;1/2)$ was obtained by Ben-Artzi and Klainerman \cite{ben1992decay} for $n \ge 3$, Chihara \cite{chihara2002smoothing} for $n=2$ and Kenig--Ponce--Vega \cite{kenig1991oscillatory} for $n=1$. For other dispersive equations, the global smoothing estimates $K_{glb}(L_{t,x}^2;\alpha)$ are similarly obtained. (For details, see e.g. \cite{ruzhansky2012smoothing} and the references therein.) 

Theorem \ref{thm:LocToGlb} not only gives another proof for the global smoothing estimates $K_{glb}(L_{t,x}^q;\alpha)$ for $q>2$ (except the endpoint), but also resolves the sharpness issue of the local smoothing estimates $K_{loc}(L_{t,x}^2;\alpha)$. The sharpness of the known global smoothing estimates $K_{glb}(L_{t,x}^2;\alpha)$ is already settled, but that of the local smoothing estimates $K_{loc}(L_{t,x}^2;\alpha)$ is not. For instance, in the Schr\"odinger equation it is known that if $\alpha > 1/2$, the global smoothing $K_{glb}(L_{t,x}^2;\alpha)$ fails (see, e.g., \cite{ruzhansky2012smoothing}), but the fact that if $\alpha > 1/2$ the local smoothing $K_{loc}(L_{t,x}^2;\alpha)$ also fails was recently shown by Sun, Tr\'elat, Zhang and Zhong \cite{sun2017sharpness} when $n=1$. (In fact, Sun, Tr\'elat, Zhang and Zhong considered some class of dispersive equations.) By Theorem \ref{thm:LocToGlb} we can see that this fact holds in higher dimensions.
Generally we have the following.

\begin{cor}
Assume that $\Phi$ satisfies the condition \eqref{homog_cond}. 
Suppose that for each $r \ge 2$ there is an $\alpha(r) \in \mathbb R$ such that the global-in-time estimate $K_{glb}(L_{t,x}^r;\alpha)$ holds for $\alpha \le \alpha(r)$ but fails for $\alpha > \alpha(r)$. Then, the local-in-time smoothing estimate $K_{loc}(L_{t,x}^r;\alpha)$ also holds for $\alpha \le \alpha(r)$ but fails for $\alpha > \alpha(r)$.
%
%
\end{cor}

\begin{proof}
It is obvious that $K_{loc}(L_{t,x}^r;\alpha)$ holds for $\alpha \le \alpha(r)$. 
Observe that the graph $\{(s, \alpha(1/s)) : 0 \le s \le 1/2\}$ is continuous (and convex upward) by interpolation.
Assume for contradiction that the local-in-time smoothing estimate $K_{loc}(L_{t,x}^r;\alpha)$ holds for some $\alpha > \alpha(r)$. 
Then from (i) of Theorem \ref{thm:LocToGlb} and H\"older's inequality it follows that $K_{glb}(L_{t,x}^{\tilde r}; \tilde \alpha)$ holds for $\tilde r > r$ and $\tilde \alpha < \alpha - n(\frac{1}{r} - \frac{1}{\tilde r})$. By the supposition one has $\alpha - n(\frac{1}{r} - \frac{1}{\tilde r}) \le \alpha(\tilde r)$, and so $\alpha \le \alpha(r)$ as $\tilde r \to r$. But it contradicts the supposition  $\alpha > \alpha(r)$. Thus $K_{loc}(L_{t,x}^r;\alpha)$ fails for $\alpha>\alpha(r)$. 
\end{proof}

Next we consider the local maximal estimates $K_{loc}(L_x^qL_t^\infty;\alpha)$. It is closely relevant to the pointwise convergence problem of the solution $u(t,x)$ as $t \to 0$. Carleson \cite{carleson1980some} posed the problem to determine the optimal range of $s$ for which the solution $u(t,x)$ to the Schr\"odinger equation converges to the initial $f(x)$ almost everywhere $x \in \mathbb R^n$ whenever $f \in H^{s}(\mathbb R^n)$. When $n=1$, Carleson \cite{carleson1980some} proved the convergence for the sharp range $s \ge 1/4$ through the maximal estimate $K_{loc}(L_x^{1}L_t^\infty;-1/4)$. When $n=2$, Bourgain \cite{bourgain1991some}, Moyua--Vargas--Vega \cite{moyua1996schrodinger}, Tao--Vargas \cite{tao2000bilinearII} and S. Lee \cite{lee2006pointwise} made improvements, and recently Du--Guth--Li \cite{du2017sharp} have obtained the convergence for the sharp range $s> 1/3$ by showing the local maximal $K_{loc}(L_x^{3}L_t^{\infty};\alpha)$ for $\alpha < -1/3$. In higher dimensions, Sj\"olin \cite{sjolin1987regularity}, Vega \cite{vega1988schrodinger}, Bourgain \cite{bourgain2013schrodinger} and Du--Guth--Li--Zhang \cite{du2018pointwise} made progresses, and recently, Du--Zhang  \cite{du2018sharp} have proven the convergence for the sharp range $s>n/2(n+1)$ by showing the local maximal $K_{loc}(L_x^{2}L_t^{\infty};\alpha)$ for $\alpha < -n/2(n+1)$. (For sharpness of the regularity range $s$, see \cites{dahlberg1982note, Bourgain2016, luca2017coherence, luca_rogers_2017}.) (When $n=1$ the maximal estimate $K_{loc}(L_x^{2}L_t^{\infty};\alpha)$ for $\alpha \le -1/4$ was already obtained by Kenig-Ruiz \cite{kenig1983strong}.)  

By Kenig--Ponce--Vega \cite{kenig1991oscillatory}, when $n=1$,  the global-in-time maximal estimate $K_{glb}(L_x^{4}L_t^{\infty};\alpha)$, $\alpha \le -1/4$ was proved for the Schr\"odinger equation, which also implies Carleson's convergence result. As far as we know, the global maximal estimates $K_{glb}(L_x^{q}L_t^{\infty};\alpha)$ have not been addressed in higher dimensions. By Theorem \ref{thm:LocToGlb}, from the local maximal estimates $K_{loc}(L_x^3L_t^\infty;-1/3-\epsilon)$ of Du--Guth--Li \cite{du2017sharp} and $K_{loc}(L_x^2L_t^{\infty};-\frac{n}{2(n+1)}-\epsilon) $ of Du--Zhang  \cite{du2018sharp} we have the following:

\begin{cor} \label{cor:max}
Let $u(t,x)$ be the solution to the free Schr\"odinger equation, i.e., $\Phi(\xi) = |\xi|^2$.
Then,
\begin{itemize}
\item[(i)]
For $n=2$,
$K_{glb}(L^3_xL_t^{\infty};\alpha)$ holds for all $\alpha < -1/3$,
\item[(ii)]
For $n \ge 1$, $K_{glb}(L_x^2L_t^{\infty};\alpha)$ holds for all $\alpha < - \frac{n}{2(n+1)}$.
\end{itemize}
Moreover, the ranges of $\alpha$ in the above statements are sharp except the endpoint.
\end{cor}

\begin{proof}
Consider (i).
For any small $\epsilon>0$ we set $r=1/\epsilon$ and $\tilde r=2/\epsilon$. 
By the local maximal estimate $K_{loc}(L_x^3L_t^{\infty};-1/3-\epsilon/2)$ of Du--Guth--Li \cite{du2017sharp} and H\"oler's inequality, one has $K_{loc}(L_x^3L_t^{r};-1/3-\epsilon/2)$. From this and (i) of Theorem \ref{thm:LocToGlb} we have the global-in-time estimate $K_{glb}(L_x^3L_t^{\tilde r};-1/3-2\epsilon)$, which is equivalent to
\begin{equation} \label{almaxi}
\|  \langle D \rangle^{\beta} u \|_{L_x^3(\mathbb B^2; L_t^{\tilde r}(\mathbb R))}
\le C_{\epsilon} \|f\|_{H^{1/3+2\epsilon +\beta}(\mathbb R^2)}, \qquad \forall \beta \in \mathbb R.
\end{equation}

By Sobolev embedding $W^{1/\tilde r,\tilde r}(\mathbb R) \subset L^{\infty}(\mathbb R)$,
\[
\| u \|_{L_x^3(\mathbb B^2; L_t^{\infty}(\mathbb R))}
\le C \|  u \|_{L_x^3(\mathbb B^2; W_t^{1/\tilde r, \tilde r}(\mathbb R))}.
\]
By the property of the Schr\"odinger equation that one time-derivative corresponds to two space-derivatives, the above equation is 
\[
\le C \| u \|_{W_x^{2/\tilde r,3}(\mathbb B^2; L_t^{\tilde r}(\mathbb R))}.
\]
We combine this with \eqref{almaxi} with $\beta=2/\tilde r=2$. Then we obtain
\[
\| u \|_{L_x^3(\mathbb B^2; L_t^{\infty}(\mathbb R))}
\le C_{\epsilon}  \|f\|_{H^{1/3 + 3\epsilon}(\mathbb R^2)},
\]
which implies (i). The (ii) can be shown similarly, so we leave it to the reader.
\end{proof}

Very recently, for some class of dispersive equations including our cases \eqref{homog_cond}  Cho and Ko \cite{cho2018note} obtained the same local maximal estimates with those of Du--Guth--Li and Du-Zhang. Thus Corollary \ref{cor:max} holds when $\Phi$ satisfies \eqref{homog_cond}, but in this case we do not know whether the above range of $\alpha$ is sharp or not except $m=2$.

\smallskip

The paper is organized as follows. In section 2 we first reduce the solution $u$ to a frequency localized operator $U$ by using a standard Littlewood-Paley argument. Next, we define wave-packets and derive some properties of the wave-packet decomposition. In section 3 Theorem \ref{thm:LocToGlb} is proven through two propositions. The first proposition is shown by using the wave packets. The second proposition is just stated. In section 4 we give the proof of  the second proposition in section 3.

\textit{Notation}. Throughout this paper, let $C$ denote various large constants that vary from line to line, which possibly depend on $q$, $r$, $n$ and $m$.
\section{Preliminaries}

\subsection{Reduction to frequency localized operators}

The solution $u$ has a representation of the form
\begin{equation} \label{sol_op}
u(t,x) = e^{it\Phi(D)}f(x) :=  \frac{1}{(2\pi)^n}\int_{\mathbb R^n} e^{ix \cdot \xi} e^{i t \Phi(\xi)} \hat f(\xi) d\xi
\end{equation}
for all Schwartz functions $f$, where the Fourier transform $\hat f$ is defined as 
\[
\hat f(\xi) := \int_{\mathbb R^n} e^{- i y \cdot \xi} f(y) dy.
\] 
We define a frequency localized operator $U$ by
\begin{equation} \label{defU}
U f(t,x) :=\int e^{i (x \cdot \xi + t\Phi(\xi))} \hat f(\xi) \varphi(\xi) d\xi,
\end{equation}
where $\varphi \in C_0^{\infty}$ is a bump function supported on 
\begin{equation} \label{projS} 
\Pi = \{\xi \in \mathbb R^n : 1/2 \le |\xi| \le 2, |\xi/|\xi| - e_1| \le \pi/4\},
\end{equation}
where $e_1 = (1,0,0,\cdots,0) \in \mathbb R^n$ is a standard unit vector.
From a standard Littlewood-Paley argument we have the following lemma:

\begin{lem} \label{lem:LP_rdc}
Assume the second condition of \eqref{homog_cond}. 
Let $I_{R^m} := \{ t \in \mathbb R : R^m/2 \le t \le 2R^m \}$. 
\begin{itemize}
\item[(i)] 
If $K_{loc}(L_x^qL_t^r;\alpha)$ holds,
then for any $R \ge 1$ the estimate
\begin{equation} \label{red_form}
\| U f \|_{L_x^q(B_R; L_t^r(I_{R^m}) )} \le C_\alpha R^{-\alpha+ \frac{n}{q} + \frac{m}{r} - \frac{n}{2}  } \|f\|_{L^{2}(\mathbb R^n)}
\end{equation}
holds for all $f \in L^2(\mathbb R^n)$. Inversely, if the  estimate \eqref{red_form} holds for all $R \ge 1$ and all $f \in L^2(\mathbb R^n)$, then $K_{loc}(L_x^qL_t^r;\alpha-\epsilon)$ holds for all $\epsilon>0$.
 
\item[(ii)]
In the statement $\mathrm{(i)}$, $K_{loc}(L_x^qL_t^r;\alpha)$ and the inequality \eqref{red_form} can be replaced with $K_{glb}(L_x^qL_t^r;\alpha)$ and the inequality
\[
\| U f \|_{L_x^q(B_R; L_t^r(\mathbb R) )} \le C_\alpha R^{-\alpha+ \frac{n}{q} + \frac{m}{r} - \frac{n}{2}  } \|f\|_{L^{2}(\mathbb R^n)},
\]
respectively.

\item[(iii)] 
In the statements $\mathrm{(i)}$ and $\mathrm{(ii)}$, $L_x^qL_t^r$ can be replaced with $L_t^rL_x^q$. \end{itemize}
\end{lem}
\begin{proof}
The proofs of (i), (ii) and (iii) are almost identical. So, we will give the proof of (i) only.
By commuting $\langle D \rangle^\alpha e^{it\Phi(D)}  = e^{it\Phi(D)} \langle D \rangle^\alpha$, we see that the estimate $K_{loc}(L_x^qL_t^r; \alpha)$ is equivalent to the estimate
\begin{equation} \label{secG}
\| u \|_{L_x^q(\mathbb B^n; L_t^r(\mathbb I))} \le C_\alpha \|f\|_{H^{-\alpha}(\mathbb R^n)}.
\end{equation}
To show \eqref{red_form}, we take an initial data $f(x) = R^{n}(\hat g \varphi)^{\vee}(Rx)$ in the above inequality. Then after rescaling $x \mapsto R^{-1}x$ and $t \mapsto R^{-m}t$, we can obtain \eqref{red_form}.

To show that the inequality \eqref{red_form} implies $K_{loc}(L_x^qL_t^r;\alpha)$, let us introduce some necessary things.
Let $\psi_0$ and $\psi_k$ be smooth functions supported in $\{\xi \in \mathbb R^n:|\xi| \le 2\}$ and $\{\xi \in \mathbb R^n:2^{k-1} \le |\xi| \le 2^{k+1} \}$ respectively such that $1= \psi_0 + \sum_{k=1}^{\infty} \psi_k$.
We define multipliers $S_k$ by $\widehat{S_k f} = \psi_k \hat f$ for $k=0,1,2,\cdots$, and let $\tilde S_k$ be multipliers given by a bump function adapted to $\{2^{k-1} \le |\xi| \le 2^{k+1} \}$ such that $S_k= \tilde S_k S_k$.
If we define $T_kf$ and $f_k$ as $T_k f(x,t):= e^{it\Phi(D)} \tilde S_k f(x)$ and $f_k := S_k f$ respectively, one has
\[
e^{it\Phi(D)}f(x) = \sum_{k=0}^{\infty} T_kf_k.
\]
By the triangle inequality and the Cauchy--Schwarz inequality,
\begin{align} \label{dyDec}
\Big\| \sum_{k=0}^{\infty} T_kf_k \Big\|_{L_x^q(\mathbb B^n; L_t^r(\mathbb I))} 
&\le  \Big( \sum_{k=0}^{\infty} 2^{-2\epsilon k  } \Big)^{1/2} \Big( \sum_{k=0}^{\infty}  2^{2\epsilon k  } \|  T_kf_k \|_{L_x^q(\mathbb B^n; L_t^r(\mathbb I))}^{2} \Big)^{1/2} \nonumber \\
&\le C_{\epsilon} \Big( \sum_{k=0}^{\infty}  2^{2\epsilon k  } \|  T_kf_k \|_{L_x^q(\mathbb B^n; L_t^r(\mathbb I))}^{2} \Big)^{1/2}
\end{align}
for all $\epsilon>0$.

It is easy to show that
\[ 
\|  T_0f_0 \|_{L_x^q(\mathbb B^n; L_t^r(\mathbb I))} \le C \|f_0\|_{L^2(\mathbb R^n)}.
\]
Indeed, we have
\(
\|  T_0f_0 \|_{L_x^q(\mathbb B^n; L_t^r(\mathbb I))} \le \| T_0f_0 \|_{L^\infty(\mathbb I \times \mathbb B^n)} \le \|\hat f_0 \|_{L^1(\mathbb R^n)}. 
\)
Since $\hat f_0$ is supported in the ball $B(0,2)$, by the Cauchy--Schwarz inequality this  is bounded by
$C\|\hat f_0 \|_{L^2(\mathbb R^n)}$, which equals  $C\|f_0 \|_{L^2(\mathbb R^n)}$ by Plancherel's theorem.

By the second condition \eqref{homog_cond} we have 
\begin{equation} \label{scl}
U f( t,x) = T_k[(f \ast \varphi^{\vee})(2^{k} \cdot)](2^{-mk}t,2^{-k}x).
\end{equation}
Thus by using a proper finite partitioning and \eqref{red_form},
\begin{equation*}
\|T_k f_k \|_{L_x^q(\mathbb B^n; L_t^r(\mathbb I))} \le C_\alpha 2^{-k\alpha}\|f_k\|_{L^2(\mathbb R^n)}
\end{equation*}
for $k=1,2,3,\cdots$.
We insert these estimates into \eqref{dyDec}. Then,
\begin{align*}
\Big\| \sum_{k=0}^{\infty} T_kf_k \Big\|_{L_x^q(\mathbb B^n; L_t^r(\mathbb I))} 
&\le C_{\alpha,\epsilon} \Big( \sum_{k} 2^{-2k(\alpha - \epsilon)} \|  f_k \|_{L^2(\mathbb R^n)}^{2} \Big)^{1/2} \\
&= C_{\alpha,\epsilon} \|  f_k \|_{H^{-\alpha + \epsilon}(\mathbb R^n)}.
\end{align*}
Therefore, by \eqref{secG} we have $K_{loc}(L_x^qL_t^r;\alpha-\epsilon)$ for all $\epsilon >0$.
\end{proof}

\subsection{Wave packet decomposition}

Now we introduce a wave packet decomposition. We first define some functions for partitioning.
Let $\phi$ be a bump function supported in $B(0,3/2)$ such that 
\begin{equation} \label{part1}
\sum_{j \in \mathbb Z^n} \phi^2(x-j)=1.
\end{equation}
Let $\psi$ be a Schwartz function whose Fourier transform is supported in $B(0,2/3)$ such that 
\begin{equation} \label{part2}
\sum_{j \in \mathbb Z^n} \psi^2(\xi-j)=1.
\end{equation} 
For $R \ge 1$, let $\mathcal P_R :=  R \mathbb Z^n$ and $\mathcal V_{R^{-1}} := R^{-1} \mathbb Z^n$ be lattice sets.
For each $(l,v) \in \mathcal P_{R} \times \mathcal V_{R^{-1}}$ we define $\phi_l$ and $\psi_v$ as $\phi_l(x) := \phi(\frac{x-l}{R})$ and $\psi_v(\xi) := \psi(R(\xi-v))$ respectively, and using these we define a function $f_{(l,v)}$ by
\begin{equation}\label{def:f_T}
f_{(l,v)} :=  m_{v}(\phi_{l}f),
\end{equation}
where $m_{v}$ is a multiplier defined by $\widehat{m_{v}f}(\xi) = \psi_{v}(\xi) \hat f(\xi)$. We see that $f_{(l,v)}$ is supported in $B(l,CR)$ and its Fourier transform is essentially supported in $B(v,\frac{1}{CR})$.
We have the following decomposition of $f$ as
\begin{equation} \label{waveDec} 
f = \sum_{(l,v) \in \mathcal P_{R} \times \mathcal V_{R^{-1}}}  f_{(l,v)}.
\end{equation}
The $f_{(l,v)}$ have the following properties: 
%
\begin{lem} \label{lem:wavepacket}
Let $R \ge 1$, and for each $(l,v) \in \mathcal P_R \times \mathcal V_{R^{-1}}$ let $f_{(l,v)}$ be defined as in \eqref{def:f_T}. Then,
\begin{itemize}
\item[(i)]
For $(t,x) \in I_{R^2} \times \mathbb R^n$,
\[
|Uf_{(l,v)}(t,x)| \le C_M R^{-n/2}(1+R^{-1}|(x-l)+t \nabla \Phi(v)|)^{-M} \|f_{(l,v)}\|_{L^2(\mathbb R^n)}
\]
for all $M \ge 1$.
\item[(ii)]
\[
\Big( \sum_{(l,v) \in \mathcal P_{R} \times \mathcal V_{R^{-1}}}  \|f_{(l,v)} \|_{L^2(\mathbb R^n)}^2 \Big)^{1/2} = \|f\|_{L^2(\mathbb R^n)}.  
\]
\item[(iii)]
For any sub-collection $\mathcal{\tilde P} \subset \mathcal P_{R}$ and $\mathcal{\tilde V} \subset \mathcal V_{R^{-1}}$,
\[
\Big\| \sum_{(l,v) \in \mathcal{\tilde P} \times \mathcal{\tilde V}} f_{(l,v)}\Big\|_{L^2(\mathbb R^n)}
\le C \Big( \sum_{(l,v) \in \mathcal{\tilde P} \times \mathcal{\tilde V}} \|f_{(l,v)}\|_{L^2(\mathbb R^n)}^2 \Big)^{1/2}.
\] 
\end{itemize}
\end{lem}
\begin{rem}
For each $(l,v) \in \mathcal P_{R} \times \mathcal V_{R^{-1}}$ we define a set $T_{(l,v)}$ by 
\[
T_{(l,v)}= \{ (t,x) \in \mathbb R \times \mathbb R^n : \big| (x-l) + t \nabla \Phi(v) \big| \le R \},
\]
which is the $R$-neighborhood of the line that is passing through $(0,l) \in \mathbb R \times \mathbb R^n$ and parallel to $(-1,\nabla \Phi(v))$. 
The property (i) implies that $Uf_{(l,v)}$ is essentially supported in $T_{(l,v)}$ in $I_{R^m} \times \mathbb R^n$. 
\end{rem}

\begin{proof}
Consider Property (i). 
We write as
\begin{equation*} 
Uf_{(l,v)}(t,x) 
= \int K_v(t,x-y) f_{(l,v)}(y) dy,
\end{equation*}
where the kernel $K_v$ is defined by
\begin{equation} \label{EfForm}
K_v(t,x) = \int e^{ i (x \cdot \xi + t \Phi(\xi))} \tilde \chi_{v} \varphi(\xi) d\xi
\end{equation}
where $\tilde \chi_{v}$ is a smooth function supported in a small neighborhood of the ball $B(v,\frac{2}{3R})$ such that 
$\tilde \chi_{v} = 1$ on the ball $B(v,\frac{2}{3R})$. 
Using a stationary phase method we can obtain that for $(t,x) \in I_{R^2} \times \mathbb R^n  $,
\begin{equation} \label{ker_est}
|K_v(t,x)| \le C_M R^{-n} (1+R^{-1}|x+ t \nabla\Phi(v) |)^{-M}, \qquad \forall M \ge 1.
\end{equation}
Indeed,
by change of variables $\xi \mapsto R^{-1}\xi +v$,
\[
K_v(t,x)= R^{-n}\int e^{i\Omega(t,x,\xi)} \rho(\xi) d\xi,
\]
where $\rho$ is a smooth function whose support is in $B(0,1)$, and  
\[
\Omega(t,x,\xi) = x\cdot (R^{-1}\xi +v) + t \Phi(R^{-1} \xi +v). 
\] 
We break $\Phi(R^{-1}\xi +v)$ into a Taylor series in $\xi$. Let us write as
\[
\Omega(t,x,\xi)=  x\cdot (R^{-1} \xi +v) + t \Phi(v) + R^{-1} t \nabla \Phi(v) \cdot \xi + R^{-2} t \eta_{R^{-1},v}(\xi)
\]
where
\[
\eta_{R^{-1},v}(\xi) = \xi^{T} H_\Phi(v) \xi + O(R^{-1})
\]
in $B(0,1)$.
We now have 
\begin{equation} \label{eqn:K_v}
|K_v(t,x)| = R^{-n} \bigg| \int e^{iR^{-1}(x+ t \nabla \Phi(v) ) \cdot \xi} \tilde \rho(\xi) d\xi \bigg|
\end{equation}
where
\[
\tilde \rho(\xi) := e^{iR^{-2} t\eta_{R^{-1},v}(\xi)}  \rho(\xi).
\]
Observe that if $|t| \le CR^{2}$ then for nonnegative integers $\gamma_1, \cdots, \gamma_n$, the differential function 
\(
|\partial_1^{\gamma_1} \cdots \partial_n^{\gamma_n} \tilde \rho|
\) is bounded by a constant $C(\gamma_1, \cdots, \gamma_n,v)$ in $B(0,1)$ which is independent of $R$ and $t$.
By using integration by parts with this observation (for details, see \cite{stein1993harmonic}*{Proposition 4 in Chapter VIII}) we can obtain that for $(t,x) \in I_{R^2} \times \mathbb R^n$,
\[
\bigg| \int e^{i R^{-1}  (x+ t \nabla \Phi(v) ) \cdot \xi} \tilde \rho(\xi) d\xi \bigg| \le C_{M}(1+R^{-1}|x+ t \nabla\Phi(v) |)^{-M}, \qquad \forall M \ge 1.
\]
Inserting this estimate into \eqref{eqn:K_v} we thus have \eqref{ker_est}.

From \eqref{ker_est} it follows that for $(t,x) \in   I_{R^2} \times \mathbb R^n$,
\begin{align}
  |U f_{(l,v)}(t,x)| &\le C_M R^{-n} (1+R^{-1}|x-l + t \nabla\Phi(v) |)^{-M}
\int | f_{(l,v)}(y) | dy \nonumber\\
&\le C_M R^{-n/2} (1+R^{-1}|x-l + t \nabla\Phi(v) |)^{-M} \| f_{(l,v)} \|_{L^2(\mathbb R^n)}
\label{EfTest}
\end{align}
for any $M \ge 1$, where the Cauchy-Schwarz inequality is used in the last line.
Thus, Property  (i) is obtained.

\smallskip

Consider Property (ii). 
By Plancherel's theorem and \eqref{part2},
\[
\sum_{l} \sum_{v} \int |f_{(l,v)}|^2 =
\sum_{l} \sum_{v} \int |\psi_{v} \widehat{\phi_l f}|^2 
= \sum_{l}  \int |\widehat{\phi_l f}|^2.
\]
We use Plancherel's theorem again, and by \eqref{part1} the above equation equals 
\[
\sum_{l}  \int |\phi_l f|^2
= \int |f|^2.
\]
Thus we have Property (ii).

\smallskip

Consider Property (iii). By Plancherel's theorem and a partition of unity $\{\psi_v\}$,
\[
\int \Big|\sum_{v\in \mathcal{\tilde V}} \sum_{l\in \mathcal{\tilde P}} f_{(l,v)} \Big|^2 \le C
\sum_{v\in \mathcal{\tilde V}} \int \Big|\psi_{v} \sum_{l\in \mathcal{\tilde P}}\widehat{\phi_l f} \Big|^2. 
\]
By Plancherel's theorem, the right side of the above equation equals 
\[
C \sum_{v\in \mathcal{\tilde V}} \int \Big|\sum_{l\in \mathcal{\tilde P}} f_{(l,v)} \Big|^2.
\]  
Since $f_{(l,v)}$ is supported in $B(l,CR)$, the above equation is
\[
\le  C \sum_{v\in \mathcal{\tilde V}} \sum_{l\in \mathcal{\tilde P}} \int |  f_{(l,v)} |^2.
\]
Thus, we have Property (iii).
\end{proof}
%

\section{From local-in-time to global-in-time}
In this section we prove Theorem \ref{thm:LocToGlb}.
From Lemma \ref{lem:LP_rdc} we see that Theorem \ref{thm:LocToGlb} is reduced to the following:
\begin{thm} \label{thm:G2}
Let $2 \le q,r <\infty$ and $R \ge 1$.
\begin{itemize}
\item[(i)]
Suppose that the estimate
\begin{equation} \label{starInd}
\| U f \|_{L_x^q(B_R; L_t^r(I_{R^{m}}) )} \le  A(R) \|f\|_{L^{2}(\mathbb R^n)}
\end{equation}
holds for all $f \in L^2(\mathbb R^n)$. Then for any $\tilde r>r$, the estimate
\begin{equation} \label{eqn:gg}
\| U f \|_{L_x^{q}(B_R; L_t^{\tilde r}(\mathbb R) )} 
\le C_{\tilde r} R^{n(\frac{1}{r} - \frac{1}{\tilde r})}  A(R) \|f\|_{L^{2}(\mathbb R^n)}
\end{equation} 
holds for all $f \in L^2(\mathbb R^n)$, 
where  
\(
A(R) := C_{\alpha} R^{-\alpha+ \frac{n}{q}+ \frac{m}{r} -\frac{n}{2} }.
\)

\item[(ii)]
If the estimate \eqref{starInd} with $L_t^rL_x^q$ replacing $L_x^qL_t^r$ holds for all $f \in L^2(\mathbb R^n)$,  then for any $\tilde r > r$,
\begin{equation} \label{eqn:gg2}
\| U f \|_{L_x^{q}(B_R; L_t^{\tilde r}(\mathbb R) )} 
\le C_{\tilde r,\epsilon} R^{n(\frac{1}{r} - \frac{1}{\tilde r}) +\epsilon}  A(R) \|f\|_{L^{2}(\mathbb R^n)}
\end{equation}
for all $f \in L^2(\mathbb R^n)$ and all $\epsilon>0$.
\end{itemize}
\end{thm}
Note that the ${\epsilon}$-loss in \eqref{eqn:gg2}  is absorbed to those in $K_{glb}(L_t^rL_x^q;\alpha-\epsilon)$.
%

\subsection{Extension the time interval $I_{R^m}$ to an arbitrary finite interval $I_H$}
To prove Theorem \ref{thm:G2} we first extend the time interval $I_{R^m}$ to an arbitrary finite interval $I_H:= [H/2,2H]$.
We will use the geometric observation in \cite{lee2006pointwise}*{Lemma 2.3} as follows.
Let $\Gamma := \{(\Phi(\xi),\xi)\}$ be a surface in $\mathbb R \times \mathbb R^n$, where $\Phi$ satisfies the first condition of \eqref{homog_cond}. Then, 
for any  $\xi \in\mathbb R^n$ with $ 1/2 \le |\xi| \le 2$ the angle between the normal vector $(-1,\nabla \Phi(\xi))$ to $\Gamma$ and the $t$-axis is lager than some positive constant. 

\begin{prop} \label{prop:extIH}
Let $R \ge 1$.
\begin{itemize}
\item[(i)]
Let $2 \le q, r \le \infty$.
If one has \eqref{starInd} for all $f \in L^2(\mathbb R^n)$, then for any $H \ge 1$,
\begin{equation} \label{eqGoal}
\| U f \|_{L_x^q(B_R; L_t^r(I_H) )} 
\le C_{\epsilon} H^{\epsilon}  A(R) \|f\|_{L^{2}(\mathbb R^n)}
\end{equation}
for all $f \in L^2(\mathbb R^n)$ and all $\epsilon >0$.

\item[(ii)]
Let $1 \le q \le \infty$ and $2 \le r \le \infty$. If the estimate \eqref{starInd} with $L_t^rL_x^q$ replacing $L_x^qL_t^r$ is valid, then the the estimate \eqref{eqGoal} with $L_t^rL_x^q$ in replacement of  $L_x^qL_t^r$ holds. 
\end{itemize}
\end{prop}

\begin{proof}
Consider (i).
It suffices to show that for any $H \ge 1$,
\begin{equation} \label{Kaloc}
\| U f \|_{L_x^q(B_R; L_t^r(I_{H}) )} \le  C_\epsilon H^{\epsilon} A(R)\|f\|_{L^{2}(\mathbb R^n)} + C_M H^{-M}\|f\|_{L^{2}(\mathbb R^n)}
\end{equation}
for all $f \in L^2(\mathbb R^n)$, $\epsilon >0$ and $M \ge 1$.

We will prove \eqref{Kaloc} by induction. For $1 \le H \le R^{m}$, one has \eqref{Kaloc} from \eqref{starInd}.
We assume that for some $H \ge 1$ the estimate
\eqref{Kaloc} holds for all $f \in L^2(\mathbb R^n)$, $\epsilon > 0$ and $M \ge 1$. Now, it is enough to show \eqref{Kaloc} with $I_{H^2}$ replacing $I_H$.
Let $\{t_j\}$ be a maximal $H$-separated subset in the interval $I_{H^2}$, and let $I_j$ be the interval of length $H$ with a center at $t_j$.
If $q \le r$ then by embedding $\ell^q \subset \ell^r$,
\begin{align*} 
\| U f \|_{L_x^q(B_R; L_t^r(I_{H^2}) )}  
&\le \Big( \int_{B_R} \Big( \sum_{j} \| Uf \|_{L_t^r(I_j)}^r \Big)^{q/r} dx\Big)^{1/q} \nonumber\\
&\le  \Big( \int_{B_R} \sum_{j}   \| Uf \|_{L_t^r(I_j)}^q  dx\Big)^{1/q} \nonumber\\
&= \Big(  \sum_{j} \| Uf \|_{L_x^q(B_R; L_t^r(I_j))}^q  \Big)^{1/q}.
\end{align*}
If $q > r$ then by Minkowski's inequality,
\begin{align*} 
\| U f \|_{L_x^q(B_R; L_t^r(I_{H^2}) )}  
&\le \Big( \int_{B_R} \Big( \sum_{j} \| Uf \|_{L_t^r(I_j)}^r \Big)^{q/r} dx\Big)^{1/q} \nonumber\\
&\le  \Big(  \sum_{j} \| Uf \|_{L_x^q(B_R; L_t^r(I_j))}^r  \Big)^{1/r}.
\end{align*}
Thus,
\begin{equation} \label{eqn:othIj}
\| U f \|_{L_x^q(B_R; L_t^r(I_{H^2}) )}  
\le  \Big(  \sum_{j} \| Uf \|_{L_x^q(B_R; L_t^r(I_j))}^s  \Big)^{1/s},
\end{equation}
where $s= \min\{q,r\}$.

Let $\Delta_j := I_j \times B_{H}$ be an $H$-cube, and $K\Delta_j$ denote the $K$-dilation of $\Delta_j$ with the center of dilation at its center for $K >0$. We decompose $f$ into  $\{f_{(l,v)} : (l,v) \in \mathcal P_{H} \times \mathcal V_{H^{-1}}\}$.  
By using (i) of Lemma \ref{lem:wavepacket},\begin{multline*}
\| U f \|_{L_x^q(B_R; L_t^r(I_{H^2}) )}  
\le
\Big(  \sum_{j}  \Big\| \sum_{ (l,v) \in \mathcal W_j(H^{\epsilon})}  Uf_{(l,v)} \Big\|_{L_x^q(B_R; L_t^r(I_j))}^s  \Big)^{1/s} 
+ C_{M} H^{-M} \|f\|_{L^2(\mathbb R^n)}
\end{multline*}
for any $M \ge 1$, where  
\[
\mathcal W_j(H^{\epsilon}) := \{(l,v) \in \mathcal P_{H} \times \mathcal V_{H^{-1}}: T_{(l,v)} \cap H^{\epsilon}\Delta_j \neq \emptyset \}.
\]
Since \eqref{Kaloc} is translation invariant, by the induction hypothesis and embedding $\ell^2 \subset \ell^s$, 
\begin{equation*}
\Big(  \sum_{j}  \Big\| \sum_{ (l,v) \in \mathcal W_j(H^{\epsilon})}  Uf_{(l,v)} \Big\|_{L_x^q(B_R; L_t^r(I_j))}^s  \Big)^{1/s} 
\le
C_{\epsilon} 
H^{\epsilon} A(R)   \Big( \sum_{j} \Big\|  \sum_{ (l,v) \in \mathcal W_j( H^{\epsilon})} f_{(l,v)} \Big\|_{L^2(\mathbb R^n)}^2 \Big)^{1/2}.
\end{equation*}
By (iii) in Lemma \ref{lem:wavepacket},
\begin{align*}
\sum_{j} \Big\|  \sum_{ (l,v) \in \mathcal W_j(H^{\epsilon})} f_{(l,v)} \Big\|_{L^2(\mathbb R^n)}^2 
&\le C \sum_{j} \sum_{v} \sum_{l : T_{(l,v)} \cap  H^{\epsilon} \Delta_j \neq \emptyset} \| f_{(l,v)} \| _{L^2(\mathbb R^n)}^2.
\end{align*}
By rearranging the summations, the  right side of the above inequality equals 
\[
C\sum_{l} \sum_{v} \sum_{j : T_{(l,v)} \cap  H^{\epsilon} \Delta_j \neq \emptyset} \| f_{(l,v)} \| _{L^2(\mathbb R^n)}^2.
\]
\begin{figure}[htbp]
\begin{center}
\begin{tikzpicture}
\draw [thick] (1/8,3) -- (1/8,-3) -- (-1/8,-3) -- (-1/8,3) -- (1/8,3); 
\draw (-1.3,-2.7) node {$I_{H^2} \times B_R$};
\draw [dashed,semithick](1/1.8,1/1.8) -- (1/1.8,-1/1.8) -- (-1/1.8,-1/1.8) -- (-1/1.8,1/1.8) -- (1/1.8,1/1.8); 
\draw (-0.8 ,-0.2) node {$\Delta_j$};
\draw (3.5,-1) node {$T_{(l,v)}$};
\draw[shift={(0,1/2)},rotate=40]  (1/2,3.5) -- (1/2,-3.5) -- (-1/2,-3.5) -- (-1/2,3.5) -- (1/2,3.5);
\draw[shift={(0,1/2)},rotate=80]  (1/2,3.5) -- (1/2,-3.5) -- (-1/2,-3.5) -- (-1/2,3.5) -- (1/2,3.5);
\draw[shift={(0,-1/2)},rotate=60]  (1/2,3.5) -- (1/2,-3.5) -- (-1/2,-3.5) -- (-1/2,3.5) -- (1/2,3.5);
\end{tikzpicture}
\caption{}
\label{default}
\end{center}
\end{figure}

From the first condition of \eqref{homog_cond} it follows that the angle formed by the vector $(-1,\nabla \Phi(\xi))$ and the $t$-axis is lager than some positive constant. Thus, 
if $l$ and $v$ are given then 
the number of $j$ with $T_{(l,v)} \cap \Delta_j \neq \emptyset$ is $O(1)$, see Figure 1. By this observation the above equation is
\begin{align*}
&\le C  H^{C\epsilon} \sum_l \sum_{v} \|  f_{(l,v)} \|_{L^2(\mathbb R^n)}^2,
\end{align*}
and by (ii) of Lemma \ref{lem:wavepacket} it equals 
\[
C  H^{C\epsilon} \|  f \|_{L^2(\mathbb R^n)}^2.
\]
By combining all the above equations we obtain \eqref{Kaloc} with $I_{H^2}$ replacing $I_H$.

\smallskip

Consider (ii). The proof is similar to that of (i). Here, the condition $q \ge 2$ is not required because it is used for embedding $\ell^2 \subset \ell^q$ in the proof of $\mathrm{(i)}$. Since we easily have 
\[
\| U f \|_{L_t^r(I_{H^2}; L_x^q(B_R) )} \le 
\Big(  \sum_{j} \| Uf \|_{L_t^r(I_j; L_x^q(B_R) )}^r  \Big)^{1/r},
\]
we need only the condition $r \ge 2$ for embedding.  
\end{proof}


\subsection{Extension an arbitrary finite interval $I_H$ to $\mathbb R$}
The $H^\epsilon$-loss in the right side of \eqref{eqGoal} is obstructing to have the global-in-time estimates. If \eqref{eqGoal} is uniformly bounded for $H$ then one can have the global-in-time estimates by limiting $H \to \infty$. By virtue of dispersive properties such as \eqref{eqn:fourmea} below it is possible to eliminate the $H^\epsilon$-loss in exchange for some integrability.

\begin{prop} \label{prop:epRev} 
Let $2 \le q,r  < \infty$ and $R \ge 1$.
\begin{itemize}
\item[(i)]
Suppose that for any $H \ge 1$ the estimate \eqref{eqGoal} holds for all $f \in L^2(\mathbb R^n)$ and all $\epsilon>0$. Then for any $\tilde r >r$, the estimate
\eqref{eqn:gg} holds for all $f \in L^2(\mathbb R^n)$.
\item[(ii)]
If the estimate \eqref{eqGoal} with $L_t^rL_x^q$ in replacement of  $L_x^qL_t^r$ is valid for all $f \in L^2(\mathbb R^n)$ and all $\epsilon>0$ then for any $\tilde r > r$, the estimate \eqref{eqn:gg2} holds for all $f \in L^2(\mathbb R^n)$ and all $\epsilon>0$.
\end{itemize}
\end{prop}

The proof of proposition will be given in next section. We can easily see that the combination of Proposition \ref{prop:epRev} and Proposition \ref{prop:extIH} gives Theorem \ref{thm:G2}.

\section{Proof of Proposition \ref{prop:epRev}}
We adapt \textit{Tao's epsilon removal} arguments in \cite{tao2000Bochner}. First we rewrite Proposition \ref{prop:epRev} by using duality.
By Plancherel's theorem, we may replace $Uf$ with $U\hat f$ in \eqref{eqGoal}, which may be viewed as the adjoint operator of the Fourier restriction operator $\mathfrak Rf = \hat f \big|_S$ for a compact surface $S = \{ (\Phi(\xi), \xi) \in \mathbb R \times \mathbb R^n : \xi \in \Pi\}$ where $\Pi$ is as in \eqref{projS}. By duality, Proposition \ref{prop:epRev} is equivalent to the following:

\begin{prop} \label{thm:epsilonRem}
Let $1 < q,r \le 2$ and  $R \ge 1$. Let $d\sigma$ be the induced Lebesgue measure of the compact surface $S$.
\begin{itemize}
\item[(i)]
Suppose that for any $H \ge 1$,
\begin{equation} \label{eqn:restB}
\|\mathfrak Rf\|_{L^2(d\sigma)} \le C_\epsilon H^{\epsilon} A(R) \|f\|_{L_{x}^{q}(B_R;L_{t}^{r}(I_H))}
\end{equation}
for all intervals $I_H$ of length $H$, all $f \in L_{x}^{q}(B_R;L_{t}^{r}(I_H))$ and all $\epsilon >0$. Then for $1 \le \tilde r<r$, 
\begin{equation} \label{eqn:Glrest}
\|\mathfrak Rf\|_{L^2(d\sigma)} \le C_{\tilde r} 
R^{n(\frac{1}{\tilde r} - \frac{1}{r})} A(R) \|f\|_{L_{x}^{q}(B_R;L_{t}^{\tilde r}(\mathbb R))}
\end{equation}
for all $f \in {L_{x}^{q}(B_R;L_{t}^{\tilde r}(\mathbb R))}$.
\item[(ii)]
Suppose that the estimate \eqref{eqn:restB} with $L_t^rL_x^q$ in replacement of  $L_x^qL_t^r$ holds for all intervals $I_H$, all $f \in L_{t}^{r}(I_H;L_{x}^{q}(B_R))$ and all $\epsilon >0$. Then for $1 \le \tilde r < r$, 
\begin{equation*} 
\|\mathfrak Rf\|_{L^2(d\sigma)} \le C_{\tilde r,\delta} 
R^{n(\frac{1}{\tilde r} - \frac{1}{r})+\delta} A(R) \|f\|_{L_{t}^{\tilde r}(\mathbb R;L_{x}^{q}(B_R))}
\end{equation*}
for all $f \in L_{t}^{\tilde r}(\mathbb R;L_{x}^{q}(B_R))$ and all $\delta>0$.
\end{itemize}
\end{prop}

The proposition will be shown through two steps.

\subsection{Extension to sparse balls}
We define a sparse collection of balls. Recall that the Fourier transform of the surface measure $d\sigma$ has a decay estimate   
\begin{equation} \label{eqn:fourmea}
|\widehat{d\sigma}(\xi)| \le C(1+|\xi|)^{-\rho}, \qquad \xi \in \mathbb R^{n+1},
\end{equation}
where  $\rho=n/2$.
We set an exponent
\begin{equation} \label{def:gam}
\gamma := n/\rho =2.
\end{equation}
\begin{defn}
A collection $\{B(z_i,H)\}_{i=1}^{N}$ of $H$-balls is called $(N,H)$-\textit{sparse} if the centers $z_i$ are $(NH)^{\gamma}$ separated.
\end{defn}

Let $\phi$ be a radial Schwartz function that is positive on the ball $B(0,2/3)$, and whose Fourier transform is supported on the ball $B(0,3/2)$. 
For an $(N,H)$-sparse collection  $\{B(z_i,H)\}_{i=1}^{N}$,
we consider the corresponding collection $\big\{f_i \ast \hat \phi_i \big|_S \big\}_{i=1}^{N}$ of restricted functions to $S$ where $\phi_i(z) :=\phi(H^{-1}(z-z_i))$. 

\begin{lem}[\cite{tao2000Bochner}*{In the proof of Lemma 3.2}]  \label{lem:spase_decp}
If $1 \le p \le 2$ then
\begin{equation} \label{eqn:dep}
\Big\| \sum_{i=1}^{N} f_i \ast \hat \phi_i \big|_S  \Big\|_p \le CH^{1/p} \Big( \sum_{i=1}^{N} \|f_i\|_p^p \Big)^{1/p}
\end{equation}
for all $f_i \in L^p(\mathbb R^{n+1})$, where the constant $C$ is independent of $N$.
\end{lem}
\begin{proof}
By interpolation it suffices to show \eqref{eqn:dep} for $p=1$ and $p=2$. For $p=1$, the estimate is obtained by the triangle inequality and Fubini's theorem. For $p=2$ we write the left side of \eqref{eqn:dep} as
\begin{align*}
\Big\| \sum_{i} f_i \ast \hat \phi_i \big|_{S} \Big\|_2^2 
&= \sum_{i} \| f_i \ast \hat \phi_i \big|_{S} \|_2^2
+ \sum_{i \neq j} \int f_i \ast \hat \phi_i  \overline{f_j \ast \hat \phi_j} d\sigma \\
&=: (\mathrm I) + (\mathrm{II}).
\end{align*}

Using Schur's lemma we can obtain $\| f_i \ast \hat \phi_i \big|_{S} \|_2  \le C H^{1/2} \|f_i\|_2 $. Thus,
\begin{equation} \label{diag}
(\mathrm{I})  \le C H \sum_{i=1}^{N}\|f_i\|_2^2.
\end{equation}
Consider (II). By Parseval's identity,
\[
 \int f_i \ast \hat \phi_i  \overline{f_j \ast \hat \phi_j} d\sigma 
= 
\int \hat f_i \phi_i \overline{\hat f_j \phi_j  \ast \widehat{d\sigma}}.
\]
The right side of the above equation is bounded by 
\[
\big\|\hat f_i \phi_i^{1/2} \big\|_1 \big\|\hat f_j \phi_j^{1/2} \big\|_1\sup_{x,y} \big| \phi_i^{1/2}(x)  \phi_j^{1/2}(y)   \widehat{d\sigma}(x-y) \big|.
\]
By the Cauchy-Schwarz inequality and Plancherel's theorem,
\[
\|\hat f_i \phi_i^{1/2} \|_{1} \le C H^{\frac{n+1}{2}} \|f_i\|_{2}.
\]
Since the collection of $B(z_i,H)$ is $(N,H)$-sparse, we have  $|z_i-z_j| \ge (NH)^{\gamma}$ for $i \neq j$. By \eqref{eqn:fourmea} it implies 
\[
\sup_{x,y} \big| \phi_i^{1/2}(x)  \phi_j^{1/2}(y)   \widehat{d\sigma}(x-y) \big| \le C(NH)^{-\gamma \rho} = C(NH)^{-n}.
\]
We combine these estimates. Then,
\begin{align*}
(\mathrm{II})
\le CH N^{-n }\sum_{i=1}^{N}\sum_{j=1}^{N} \| f_i\|_{2} \| f_j\|_{2}. 
\end{align*}
By the Cauchy-Schwarz inequality, 
\[
\mathrm{(II)} \le CH N^{-n+1} \sum_{i=1}^{N} \| f_i\|_{2}^2.
\]
Therefore, from this and \eqref{diag}  we have \eqref{eqn:dep} for $p=2$.
\end{proof}

Using Lemma \ref{lem:spase_decp} we can obtain the restriction estimates \eqref{eqn:restB}  for functions $f$ supported in $(N,H)$-sparse balls. 

\begin{lem} \label{lem:sparseEst}
Let $1 \le q,r \le 2$, $R \ge 1$ and $H \ge 1$. Let $\{B_i\}_{i=1}^{N}$ be an $(N,H)$-sparse collection of $H$-balls.
\begin{itemize}
\item[(i)]
Suppose that  the estimate \eqref{eqn:restB} holds for all intervals $I_H$ of length $H$, all $f \in L_{x}^{q}(B_R;L_{t}^{r}(I_H))$ and all $\epsilon>0$. Then, the estimate
\[
\|\mathfrak Rf\|_{L^2(d\sigma)} \le C_{\epsilon} H^{\epsilon}A(R) \|f\|_{L_{x}^{q}(B_R;L_{t}^{ r}(\mathbb R))}
\]
holds for all $f$ that is supported in $\cup_{i=1}^{N} B_i$ and all $\epsilon>0$.
\item[(ii)]
The statement $\mathrm{(i)}$ is still valid even if $L_t^rL_x^q$ replaces $L_x^qL_t^r$.
\end{itemize}
\end{lem}
\begin{proof}
The statements (i) and (ii) are proven identically. So, we consider only (i).
If $f_i = \chi_{B_i} f$ then one has $\hat f_i = \hat f_i \ast \hat\phi_i$ where $\phi_i$ is defined as in Lemma \ref{lem:spase_decp}. Using this one has $\mathfrak R f_i = \hat f_i \big|_S = (\hat f_i \ast \hat \phi_i) \big|_S$. Here, since $\hat \phi_i$ is supported on the ball $B(0,1/H)$, we may restrict the support of $\hat f_i$ to a $O(1/H)$ neighborhood of the surface $S$, thus we may write as   
\[ 
\mathfrak R f_i  =  (\hat f_i \big|_{\mathcal N_{1/H} S} \ast \hat \phi_i) \big|_S
\]
where $\mathcal N_{1/H} S$ is a $O(1/H)$ neighborhood of the surface $S$. Let $\tilde{\mathfrak R}$ be a restriction operator defined by $\tilde{\mathfrak R} f = \hat f \big|_{\mathcal N_{1/H}S}$. If $f$ is supported in $\cup_{i} B_i$ then one has
\[
\mathfrak R f = \sum_{i=1}^{N} \tilde{\mathfrak R}f_i \ast \hat \phi_i \big|_{S}.
\]
By Lemma \ref{lem:spase_decp},
\[
\|\mathfrak R f \|_{L^2(d\sigma)} \le C H^{1/2} \Big( \sum_{i=1}^{N} \| \tilde{\mathfrak R} f_i \|_{L^2(\mathcal N_{1/H} S)}^2 \Big)^{1/2}.
\]
From \eqref{eqn:restB} it follows that
\[
\| \tilde{\mathfrak R} f \|_{L^2(\mathcal N_{1/H} S)} 
\le C_\epsilon H^{-1/2+\epsilon} A(R) \|f\|_{L_x^{q}(B_R;L_t^r(I_H))}.
\]
By applying this to the right side of the previous estimate, 
\[
\|\mathfrak R f \|_{L^2(d\sigma)} \le C_{\epsilon} H^{\epsilon} A(R) \Big( \sum_{i=1}^{N} \|f_i \|^2_{L_{x}^{q}(B_R;L_{t}^{r}(\mathbb R))} \Big)^{1/2}.
\]
If $1 \le r \le q \le 2$ then by embedding $\ell^r \subset \ell^{q} \subset \ell^{2}$,
\begin{align*}
\Big( \sum_{i=1}^{N} \|f_i \|^2_{L_{x}^{q}(B_R;L_{t}^{r}(\mathbb R))} \Big)^{1/2}
&\le \Big( \sum_{i=1}^{N} \|f_i \|^q_{L_{x}^{q}(B_R;L_{t}^{r}(\mathbb R))} \Big)^{1/q} \\
&=\Big(\int_{B_R}\sum_{i=1}^{N}\|f_i\|_{L_t^{r}(\mathbb R)}^{q} dx\Big)^{1/q} \\
&\le \Big(\int_{B_R} \Big(\sum_{i=1}^{N}\|f_i\|_{L_t^{r}(\mathbb R)}^{r} \Big)^{q/r} dx \Big)^{1/q} \\
&=\|f\|_{L_{x}^{q}(B_R;L_{t}^{ r}(\mathbb R))}.
\end{align*}
If $1 \le q \le r \le 2$ then by embedding $\ell^r \subset \ell^2$ and Minkowski's inequality,
\begin{align*}
\Big( \sum_{i=1}^{N} \|f_i \|^2_{L_{x}^{q}(B_R;L_{t}^{r}(\mathbb R))} \Big)^{1/2}
&\le \Big( \sum_{i=1}^{N} \|f_i \|^r_{L_{x}^{q}(B_R;L_{t}^{r}(\mathbb R))} \Big)^{1/r} \\
&\le \Big(\int_{B_R} \Big(\sum_{i=1}^{N}\|f_i\|_{L_t^{r}(\mathbb R)}^{r} \Big)^{q/r} dx\Big)^{1/q} \\
&=\|f\|_{L_{x}^{q}(B_R;L_{t}^{ r}(\mathbb R))}.
\end{align*}
Therefore, 
\[
\|\mathfrak R f \|_{L^2(d\sigma)} \le C_{\epsilon} H^{\epsilon} A(R) \|f\|_{L_{x}^{q}(B_R;L_{t}^{ r}(\mathbb R))}.
\]
\end{proof}

\subsection{Extension to $\mathbb R$}
We now remove the condition that a function $f$ is supported in the union of $(N,H)$-sparse balls in Lemma \ref{lem:sparseEst}. We utilize the following decomposition: 
\begin{lem} [\cite{tao2000Bochner}*{Lemma 3.3}]\label{lem:DecE} 
Suppose that $E$ is a finite union of cubes of sidelength $c \sim 1$. Then for any $K \ge 1$ there is a decomposition 
\[
E = \bigcup_{k=1}^{K} E_k
\]
such that each $E_k$ has $O(|E|^{1/K})$ number of  $(O(|E|), O(|E|^{\gamma^{k-1}}))$-sparse collections of $O(|E|^{\gamma^{k-1}})$-balls, and is covered by the balls in the union of these collections, where $\gamma$  defined as in \eqref{def:gam}.
\end{lem}
\begin{proof}
The sets $E_k$ are constructed as follows.
Let $E_0 = \emptyset$ and $H_0=1$. 
Start with $k=1$ and proceeding recursively. 
We set
\[
H_{k} = |E|^\gamma H_{k-1}^\gamma,
\] 
and $E_k$ to be the set of all $x \in E \setminus \cup_{j=1,2,\cdots,k-1} E_j$ such that
\[
|E \cap B(x, H_k)| \le |E|^{k/K}.
\]
Then $E = \bigcup_{k=1}^{K} E_k$, and
each $E_k$ has the property described in the statement. For more details, see the proof of Lemma 3.3 in \cite{tao2000Bochner}.
\end{proof}
Roughly the above lemma says that a set $E$ can be considered as the union of sparse balls.  The constant $K$ will be a constant depending on  $\epsilon$ and $\gamma$, and the number of sparse collections, $O(|E|^{1/K})$, causes a loss in integrability.
 
Now we give the proof of Proposition \ref{thm:epsilonRem}.
\begin{proof}[Proof of Proposition \ref{thm:epsilonRem}]
Consider (i).
By interpolation,
it suffices to show that for $1 \le \tilde r < r$,
the restricted type estimate 
\begin{equation} \label{weaktype}
\|\mathfrak R \chi_E \|_{L^2(d\sigma)} \le C_{\tilde r}  R^{n(\frac{1}{\tilde r} - \frac{1}{r})} A(R) \|\chi_E\|_{L^{q}(B_R;L^{\tilde r}(\mathbb R))}
\end{equation}
for all subset $E$ in $\mathbb R \times B_R$. (For details, see \cite{fernandez1977lorentz}*{Proposition 9.2}). 
Since the surface $S$ is compact,  
$\chi_E$ can be replaced with $\chi_E \ast \varphi$ where $\varphi$ is the bump function supported on the cube of sidelength $c \sim 1$ such that $\hat \varphi$ is positive on $S$. Thus we may  further assume that $E$ is  the union of $c$-cubes.

Let $\mathrm{proj}_x E$ be the projection of $E$ onto the $x$-space. For each $x \in c \mathbb Z^n \cap \mathrm{proj}_x E$ we define $E_x$ to be the union of $c$-cubes in $E$ that intersects $\mathbb R \times \{x\}$, and let $E(h)$ be the union of $E_x$ such that $E_x$ contains $O(h)$ number of $c$-cubes.
Then one has a decomposition
\[
E= \bigcup_{h} E(h)
\]
where $h \ge 1$ are dyadic numbers.

Using Lemma \ref{lem:DecE} we decompose $E(h)$ into $
E_k$'s, and apply Lemma \ref{lem:sparseEst} to each $\chi_{E_k}$. Then,
\[
\|\mathfrak R \chi_{E_k} \|_{L^2(d\sigma)} \le C_{\epsilon} |E(h)|^{1/K} (|E(h)|^{\gamma^{k-1}})^{\epsilon} A(R) \|\chi_{E_k} \|_{L_x^q(B_R ; L_t^r(\mathbb R))},
\] 
and so
\[
\|\mathfrak R \chi_{E(h)} \|_{L^2(d\sigma)} \le C_{\epsilon} K|E(h)|^{1/K} (|E(h)|^{\gamma^{K}})^{\epsilon} A(R) \|\chi_{E(h)} \|_{L_x^q(B_R ; L_t^r(\mathbb R))}.
\]  
Since $|E(h)| \sim h |\mathrm{proj\,} E(h)|$, the right side of the above estimate is bounded by
\[
C_{\epsilon} K A(R)  h^{\frac{1}{r} + \delta} |\mathrm{proj}_x E(h)|^{\frac{1}{q} + \delta},
\]
where $\delta =\frac{1}{K} + \epsilon \gamma^{K}$.
Note that if we take $K= C_\gamma^{-1} \log(1/\epsilon)$ for a constant $C_\gamma > \log \gamma$, then one has $\delta =\frac{1}{K} + \epsilon \gamma^{K} \to 0$ as $\epsilon \to 0$.
By the triangle inequality,
\begin{align*}
\|\mathfrak R \chi_E \|_{L^2(d\sigma)} 
&\le C_\delta  A(R) 
\sum_{h} h^{\frac{1}{r} + \delta} |\mathrm{proj}_x E(h)|^{\frac{1}{q} + \delta} \\
&=C_\delta  A(R) \sum_{h} h^{-\delta_1} h^{\frac{1}{r} + \delta + \delta_1} |\mathrm{proj}_x E(h)|^{\frac{1}{q} + \delta}
\end{align*}
for any $\delta_1>0$, where the constant $K$ is absorbed to $C_\delta$.
Let us set $\tilde r$ by
$\frac{1}{\tilde r} = \frac{1}{r} + \delta  + \delta_1$.
Since $\mathrm{proj}_x E(h) \subset B_R$, one has
\begin{align*}
h^{\frac{1}{r} + \delta + \delta_1} |\mathrm{proj}_x E(h)|^{\frac{1}{q} + \delta} 
&\le CR^{n\delta}\|\chi_{E}\|_{L_x^{q}(B_R;L_t^{\tilde r}(\mathbb R))} \\
&\le CR^{n(\frac{1}{\tilde r} - \frac{1}{r})}\|\chi_{E}\|_{L_x^{q}(B_R;L_t^{\tilde r}(\mathbb R))}.
\end{align*} 
Thus we have \eqref{weaktype}.

\smallskip

Consider (ii). As in the proof of (i), it is suffices to show that for $1 \le \tilde r < r$,
\begin{equation} \label{eqn:Glrest2}
\|\mathfrak R \chi_E \|_{L^2(d\sigma)} \le C_{\tilde r} A(R) R^{n(\frac{1}{\tilde r} - \frac{1}{r}) + \delta_1} \|\chi_E\|_{L_t^{\tilde r}(\mathbb R;L_x^{q}(B_R))}
\end{equation}
for all $\delta_1>0$.
We define $\mathrm{proj}_t E$ to be the projection of $E$ onto the $t$-space, and for each $t \in c \mathbb Z \cap \mathrm{proj}_t E$ we define $E_t$ to be the union of $c$-cubes in $E$ that intersects $\{t\} \times B_R$. If $E(h)$ is the union of $E_t$ such that $E_t$ contains $O(h)$ number of $c$-cubes. Following the arguments of (i) we can get
\[
\|\mathfrak R \chi_{E(h)} \|_{L^2(d\sigma)} \le C_\delta A(R) h^{\frac{1}{q} + \delta} |\mathrm{proj}_t\, E(h)|^{\frac{1}{r} + \delta}.
\]
By the triangle inequality,
\[
\|\mathfrak R \chi_{E} \|_{L^2(d\sigma)} 
\le C_\delta A(R) \sum_{h} h^{\frac{1}{q} + \delta} |\mathrm{proj}_t\, E(h)|^{\frac{1}{r} + \delta}. 
\]
We set $\tilde r$ by $\frac{1}{\tilde r} = \frac{1}{r} + \delta $. By using $h \le CR^{n}$ the above equation is 
\begin{align*}
&\le C_{\tilde r} A(R) R^{n(\frac{1}{\tilde r} - \frac{1}{r} + \delta_1)} \sum_{h} h^{-\delta_1}h^{\frac{1}{q}} |\mathrm{proj}_t\, E(h)|^{\frac{1}{\tilde r}} \\
&\le C_{\tilde r,\delta_1} A(R) R^{n(\frac{1}{\tilde r} - \frac{1}{r} + \delta_1)}  \|\chi_E\|_{L_t^{\tilde r}(\mathbb R ; L_x^{q}(B_R))}
\end{align*}
for all $\delta_1>0$.
Thus we have \eqref{eqn:Glrest2}.

\end{proof}


\section*{Acknowledgment}
The author is indebted to the anonymous referee whose comments helped clarify some matters related to the $H^\epsilon$-loss, and led to the introduction of Proposition 3.3. 



\



\end{document}